\documentclass[a4paper,12pt]{amsart}
\usepackage{amssymb}
\usepackage{url}
\usepackage[T1]{fontenc}

\input xypic
\xyoption{all}
\xyoption{arc}

\newtheorem{theo}{Theorem}[section]

\def\remark#1{{\refstepcounter{theo}\label{#1}\noindent\sc Remark  
\arabic{section}.\arabic{theo} - }}
\def\example#1{{\refstepcounter{theo}\label{#1}\noindent\sc Example 
\arabic{section}.\arabic{theo} - }}

\def\equat{\refstepcounter{theo}$$~}
\def\endequat{\leqno{\boldsymbol{(\arabic{section}.\arabic{theo})}}~$$}

    \def\CM{{\mathbb{C}}}
  \def\dG{{\mathfrak d}}  
    
\def\FG{{\mathfrak F}}    
    
  \def\hG{{\mathfrak h}}

  \def\lG{{\mathfrak l}}  
    
    \def\NM{{\mathbb{N}}}
    
  \def\pG{{\mathfrak p}}  
    \def\QM{{\mathbb{Q}}}
    \def\RM{{\mathbb{R}}}
  \def\sG{{\mathfrak s}}

    \def\ZM{{\mathbb{Z}}}

  \def\ab{{\mathbf a}}  
    \def\BC{{\mathcal{B}}}
    \def\CC{{\mathcal{C}}}
  \def\db{{\mathbf d}}  \def\DC{{\mathcal{D}}}
    
    \def\FC{{\mathcal{F}}}
    
    \def\HC{{\mathcal{H}}}

    \def\LC{{\mathcal{L}}}
    \def\MC{{\mathcal{M}}}
    \def\NC{{\mathcal{N}}}
    \def\OC{{\mathcal{O}}}
    \def\PC{{\mathcal{P}}}
  \def\qb{{\mathbf q}}  
    \def\RC{{\mathcal{R}}}
  \def\sb{{\mathbf s}}  \def\SC{{\mathcal{S}}}
    \def\TC{{\mathcal{T}}}

  \def\crm{{\mathrm{c}}}

          \def\mba{{\bar{m}}}

          \def\sba{{\bar{s}}}

          \def\fov{{\overline{f}}}
          
          \def\hov{{\overline{h}}}

          \def\uov{{\overline{u}}}

          \def\qbt{{\tilde{\qb}}}

\def\g{\gamma}
\def\G{\Gamma}
\def\d{\delta}
\def\D{\Delta}

\def\l{\lambda}
\def\L{\Lambda}

\def\s{\sigma}

\def\th{\theta}

\def\Delb{{\boldsymbol{\Delta}}}

\def\lamb{{\boldsymbol{\lambda}}}

        \def\thet{{\tilde{\th}}}

\DeclareMathOperator{\Hom}{{\mathrm{Hom}}}

\DeclareMathOperator{\Irr}{{\mathrm{Irr}}}
\DeclareMathOperator{\Ker}{{\mathrm{Ker}}}

\DeclareMathOperator{\rad}{{\mathrm{rad}}}
\DeclareMathOperator{\Rad}{{\mathrm{Rad}}}

\def\to{\rightarrow}
\def\longto{\longrightarrow}

\def\fonction#1#2#3#4#5{\begin{array}{rccc}
{#1} : & {#2} & \longto & {#3} \\
& {#4} & \longmapsto & {#5} 
\end{array}}

\def\fonctio#1#2#3#4{\begin{array}{ccc}
{#1} & \longto & {#2} \\
{#3} & \longmapsto & {#4} 
\end{array}}

\def\vide{\varnothing}

\def\DS{\displaystyle}
\def\SS{\scriptstyle}
\def\SSS{\scriptscriptstyle}

\def\finl{~$\SS \square$}

\def\lexp#1#2{\kern\scriptspace\vphantom{#2}^{#1}\kern-\scriptspace#2}
\def\le{\hspace{0.1em}\mathop{\leqslant}\nolimits\hspace{0.1em}}
\def\ge{\hspace{0.1em}\mathop{\geqslant}\nolimits\hspace{0.1em}}

\mathchardef\inferieur="321E
\mathchardef\superieur="321F

\def\eqna{\begin{eqnarray*}}
\def\endeqna{\end{eqnarray*}}

\def\itemth#1{\item[${\mathrm{(#1)}}$]}

\DeclareMathOperator{\groth}{{\mathrm{\RC}}}
\DeclareMathOperator{\Bip}{{\mathrm{Bip}}}
\def\emptysetb{\boldsymbol{\emptyset}}
\newcommand{\Ue}{{\mathcal{U}_{v}(\widehat{\mathfrak{sl}_e})}}
\newcommand{\Uepr}{{\mathcal{U}_{v}(\widehat{\mathfrak{sl}_e})'}}

\def\pre#1{\hspace{0.1em}\mathop{\leqslant}_{#1}\nolimits\hspace{0.1em}}
\def\cell#1{\hspace{0.1em}\mathop{\sim}_{#1}\nolimits\hspace{0.1em}}
\def\SDT{\SC\DC\TC}
\def\SBT{\SC\BC\TC}

\def\dominance#1{\hspace{0.1em}\mathop{\trianglelefteq}_{#1}\nolimits
\hspace{0.1em}}
\def\dominancepr{\hspace{0.1em}\mathop{\trianglelefteq}_{\infty}^\prime\nolimits
\hspace{0.1em}}

\catcode`\@=11

\long\def\@car#1#2\@nil{#1}
\long\def\@first#1#2{#1}
\long\def\@second#1#2{#2}
\long\def\ifempty#1{\expandafter\ifx\@car#1@\@nil @\@empty
  \expandafter\@first\else\expandafter\@second\fi}
\catcode`\@=12

\def\ordre{{\SSS{\leqslant}}}

\begin{document}

\baselineskip=16pt

\title{Cellular structures on \\ Hecke algebras of type ${\boldsymbol{B}}$}

\author{C\'edric Bonnaf\'e}
\author{Nicolas Jacon}
\address{\noindent 
Labo. de Math. de Besan\c{c}on (CNRS: UMR 6623), 
Universit\'e de Franche-Comt\'e, 16 Route de Gray, 25030 Besan\c{c}on
Cedex, France} 

\makeatletter
\email{cedric.bonnafe@univ-fcomte.fr, nicolas.jacon@univ-fcomte.fr}

\makeatother

\subjclass{According to the 2000 classification:
Primary 20C08; Secondary 20C05, 05E15}

\date{\today}

\begin{abstract} 
The aim of this paper is to 
gather and (try to) unify several approaches for the modular 
representation theory of Hecke algebras of type $B$. We attempt to explain 
the connections between Geck's cellular structures (coming from Kazhdan-Lusztig 
theory with unequal parameters) and Ariki's 
Theorem on the canonical basis of the Fock spaces. 
\end{abstract}

\maketitle

\pagestyle{myheadings}

\markboth{\sc C. Bonnaf\'e and N. Jacon}{\sc Cellular structures on 
Hecke algebras of type $B$}

\centerline{\it This paper is dedicated to Gus Lehrer, on his sixtieth birthday.}

\tableofcontents

\setcounter{section}{0}

\section*{Introduction}

\medskip

The modular representation theory of Hecke algebras of type $B$ was 
first studied by Dipper-James-Murphy \cite{DJMU}: one of their  essential
tools was to construct a family of modules (called {\it Specht modules}) 
playing the same role as Specht modules in type $A$. Each of these new Specht 
modules have a canonical quotient which is zero or simple: one of the main 
problem raised by this construction is to determine which one are non-zero. 
Later, Graham and Lehrer \cite{GL} 
developed the theory of {\it cellular algebras}, which contains, as a 
particular case, the construction of Dipper-James-Murphy. 
The problem of parametrizing the simple modules and computing 
 the decomposition matrix of Specht modules were then solved by Ariki 
\cite{arikillt} using the canonical basis of Fock spaces of higher level. 
In fact, Ariki's Theorem 
provides different parametrizations of the simple modules of the Hecke 
algebra: only one of them ({\it asymptotic case}) has an interpretation 
in the framework of Dipper-James-Murphy and Graham-Lehrer. 
Recently, Geck showed that the Kazhdan-Lusztig theory with unequal 
parameters should provide a cell datum for each choice of a 
{\it weight function} on the Weyl group (if Lusztig's conjectures (P1)-(P15) 
hold \cite[Conjecture 14.2]{lusztig}). 

Our main aim in this paper is to present an overview of all these results, 
focusing particularly on conjectural connections between Uglov's point of view 
on the Fock space theory and Geck cellular structures. This should 
(if Lusztig's conjectures (P1)-(P15) hold in type $B$) lead to a unified 
approach for a better understanding of the representation theory 
of Hecke algebras. As a by-product, we should get an interpretation 
of {\it all} Ariki's parametrizations of simple modules. 

More precisely, if $Q$ and $q$ are two indeterminates, if $\HC_n$ 
denotes the Hecke $A$-algebra with parameters $Q$ and $q$ 
(here, $A=\ZM[Q,Q^{-1},q,q^{-1}]$), if $\xi$ is a positive irrational 
number (!) and if $r$ denotes the unique natural number such that 
$r \le \xi < r+1$, then Kazhdan-Lusztig theory {\it should} provide a cell 
datum $\CC^\xi = ((\Bip(n),\dominance{r}),\SBT,C^\xi,*)$ where
\begin{itemize}
\item[$\bullet$] $\Bip(n)$ is the set of bipartitions of $n$ and 
$\dominance{r}$ is a partial order on $\Bip(n)$ depending on $r$ 
(see \S\ref{sub: order});

\item[$\bullet$] If $\l \in \Bip(n)$, $\SBT(\l)$ denotes 
the set of standard bitableaux of (bi-)shape $\l$ (filled with $1$,\dots, $n$);

\item[$\bullet$] If $S$ and $T$ are two standard bitableaux of size $n$ 
and of the same shape, $C_{S,T}^\xi$ is an element of $\HC_n$ coming 
from a Kazhdan-Lusztig basis of $\HC_n$ (it heavily depends on $\xi$);

\item[$\bullet$] $* : \HC_n \to \HC_n$ is the $A$-linear anti-involution 
of $\HC_n$ sending the element $T_w$ of the standard basis to $T_{w^{-1}}$;
\end{itemize}
(see \cite[Conjecture C]{BGIL}). If this conjecture holds then, by the 
general theory of cellular algebras, we can associate to each bipartition 
$\l$ of $n$ a {\it Specht module} $S_\l^\xi$ endowed with a bilinear 
form $\phi_\l^\xi$. If $K$ is the field of fractions of $A$, then 
$KS_\l^\xi = K \otimes_A S_\l^\xi$ is the simple $K\HC_n$-module 
associated to $\l$ \cite[Th. 10.1.5]{gecklivre}. Now, if $Q_0$ and $q_0$ 
are two elements of $\CM^\times$ then, through the specialization 
$Q \mapsto Q_0$, $q \mapsto q_0$, we can construct the $\CM\HC_n$-module
$$D_\l^\xi = \CM S_\l^\xi/\Rad(\CM\phi_\l^\xi).$$
By the general theory of cellular algebras, it is known 
that the non-zero $D_\l^\xi$ give a set of representatives of 
simple $\CM\HC_n$-modules. At this stage, it must be noticed that, 
even if $KS_\l^\xi \simeq KS_\l^{\xi'}$ (where $\xi'$ is another positive 
irrational number), it might happen that $S_\l^\xi \not\simeq S_\l^{\xi'}$ 
and $D_\l^\xi \not\simeq D_\l^{\xi'}$ (it is probable that 
$(S_\l^\xi,\phi_\l^\xi) \simeq (S_\l^{\xi'},\phi_\l^{\xi'})$ for all $\lambda\in  \Bip(n)$
if and only if we have also $r \le \xi' < r+1$). 

On the other hand, if we assume further that $q_0^2$ is a primitive 
$e$-th root of unity, if $Q_0^2=-q_0^{2d}$ for some $d \in \ZM$ 
(which is only well-defined modulo $e$), and if $s=(s_0,s_1) \in \ZM^2$ 
is such that $s_0-s_1 \equiv d \mod e$, then Ariki's Theorem 
provides a bijection between the set of Uglov's bipartitions 
$\Bip_e^s(n)$ and the set of simple $\CM\HC_n$-modules. 
Moreover, the decomposition matrix is given by 
$\bigl(d_{\l\mu}^s(1)\bigr)_{\l \in \Bip(n), \mu \in \Bip_e^s(n)}$, where $\bigl(d_{\l\mu}^s(q)\bigr)_{\l,\mu \in \Bip(n)}$ is the 
transition matrix between the standard basis and Uglov-Kashiwara-Lusztig's 
canonical basis of the Fock space (see \S\ref{sub: uglov}). 
The first result of this paper insures that \cite[Conjecture C]{BGIL} 
is ``compatible'' with Ariki's Theorem in the following sense:

\bigskip

\noindent{\bf Theorem.} 
{\it Assume that \cite[Conjecture C]{BGIL} holds and assume that 
$s_0-s_1 \equiv d \mod e$ and $s_0-s_1 \le r < s_0-s_1+e$. Then
$D_\l^\xi\neq 0$ if and only $\l \in \Bip_e^s(n)$ and, 
if $\l \in \Bip(n)$ and $\mu \in \Bip_e^s(n)$, then 
$$[\CM S_\l^\xi:D_\mu^\xi] = d_{\l\mu}^s(1).$$
In particular, we have}
$$[\CM S_\l^\xi:D_\mu^\xi]\neq 0 \Longrightarrow  \l \dominance{r} \mu$$

\bigskip

Note that, in the {\it asymptotic case} (in other words, if $\xi > n-1$), then 
\cite[Conjecture C]{BGIL} holds (see \cite{geck-iancu}) and the cellular 
datum $\CC^\xi$ is more or less equivalent to the one constructed 
by Dipper, James and Mathas (see the work of Geck, Iancu and Pallikaros 
\cite{GIP}). 

\bigskip

One of the problems raised by the previous theorem (in fact, 
essentially by Ariki's Theorem) is the following: 
if $s'=(s_0',s_1') \in \ZM^2$ is such that 
$s_0'-s_1' \equiv s_0-s_1 \mod e$, then the sets $\Bip_e^s(n)$ and 
$\Bip_e^{s'}(n)$ are in bijection. Our second result (see Theorem \ref{bijection})  is to 
construct this bijection by means of an isomorphism between 
the crystals associated to the simple sub-modules $\MC[s]$ and $\MC[s']$ 
of the Fock spaces (see \S \ref{Fock} for the definition of these modules). 
This shows that the ``abstract crystal'' of an irreducible highest 
weight module is canonically associated 
to the representation theory of Hecke algebras of type $B$. By contrast, 
each realization of this crystal is associated with a natural 
parametrization of the simple modules.

Once we accept \cite[Conjecture C]{BGIL}, it is natural 
to ask whether  the matrix $\bigl(d_{\l\mu}^s(q)\bigr)$ can be interpreted as a $q$-decomposition 
matrix using Jantzen's filtration (see our Conjecture C in \S\ref{sub: jantzen}). 
It would also be interesting to see if it should be possible to construct 
different {\it Schur algebras} of type $B$ directly from Kazhdan-Lusztig's theory 
(for the asymptotic case, this construction should give rise to an algebra 
which is Morita equivalent to the two Schur algebras of type $B$ constructed 
by Du and Scott \cite{duscott} and Dipper, James and Mathas \cite{DJMa}). 
If so, it is then natural to ask for a generalization of Varagnolo-Vasserot 
Theorem for Hecke algebras of type $A$ as well as  a generalization 
of Yvonne's Conjecture \cite[Conj. 2.13]{yvonne}. Note that a construction of 
{\it Schur algebras} of type $B$ is provided by the theory of Cherednik 
algebras \cite{ggor}, but this does not provide a {\it generic} Schur algebra. 

If the reader wants more arguments for \cite[Conjecture C]{BGIL}, 
he or she is encouraged to read the original source of this conjecture 
\cite{BGIL}. Note also that recent works by Gordon and Martino 
(see \cite{gordon}, \cite{gordon martino}) 
 on  Cherednik algebras show other 
compatibilities between this conjecture and the geometry of the 
Calogero-Moser spaces, as well as with Baby Verma modules of 
Cherednik algebras {\it at $t=0$}.

Finally, it is natural to ask whether there exist different 
cellular structures for Ariki-Koike algebras associated 
to the complex reflection groups $G(d,1,n)$, these cellular 
structures being indexed by $d$-cores (or by $d$-uples of elements 
of $\ZM$): however, in this case, no Kazhdan-Lusztig's theory is 
available at that time (will there be one in the future?) 
so we have no candidate for the different cellular bases.

This paper is organized as follows. 
In the first section, we present the setting of our problem. 
Then, we recall Ariki's Theorem which allows to compute the decomposition 
matrices for Hecke algebras of type $B_n$ using objects coming from 
quantum group theory. In the third section, we introduce the combinatorial 
notions that we need to describe the Kazhdan-Lusztig theory in type $B_n$. 
The last two sections are devoted to the main results of our paper. Under 
some conjectures, we show the existence of several deep connections 
between Kazhdan-Lusztig theory and the canonical basis theory for $\Ue$.

\newpage

\section{Notation}

\medskip

Let $(W_n,S_n)$ be a Weyl group of type $B_n$ and assume that 
the elements of $S_n$ are denoted by $t$, $s_1$,\dots, $s_{n-1}$ 
in such a way that the Dynkin diagram is given by
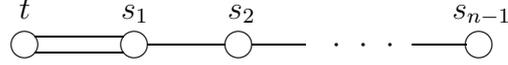
\begin{center}
\begin{picture}(220,30)
\put( 40, 10){\circle{10}}
\put( 44,  7){\line(1,0){33}}
\put( 44, 13){\line(1,0){33}}
\put( 81, 10){\circle{10}}
\put( 86, 10){\line(1,0){29}}
\put(120, 10){\circle{10}}
\put(125, 10){\line(1,0){20}}
\put(155,  7){$\cdot$}
\put(165,  7){$\cdot$}
\put(175,  7){$\cdot$}
\put(185, 10){\line(1,0){20}}
\put(210, 10){\circle{10}}
\put( 38, 20){$t$}
\put( 76, 20){$s_1$}
\put(116, 20){$s_2$}
\put(200, 20){$s_{n{-}1}$}
\end{picture}
\end{center}
The length function with respect to $S_n$ will be denoted by 
$\ell : W_n \to \NM=\{0,1,2\dots\}$. 

Let $\G$ be a torsion-free finitely generated 
abelian group and let $A$ denote the group ring 
$\ZM[\G]$. The group law on $\G$ will be denoted additively 
and we shall use an exponential notation for $A$: more precisely, 
$A=\DS{\mathop{\oplus}_{\g \in \G}} \ZM e^\g$, where $e^\g e^{\g'}=e^{\g+\g'}$ 
for all $\g$, $\g' \in \G$. We also fix two elements $a$ and $b$ in $\G$ 
and we denote by $\HC_n$ the Hecke algebra of $W_n$ over $A$ associated 
to the choice of parameters $t \mapsto b$ and $s_i \mapsto a$ 
for $1 \le i \le n-1$ symbolized by the following diagram:
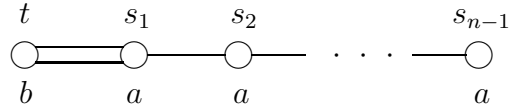
\begin{center}
\begin{picture}(220,42)
\put( 40, 20){\circle{10}}
\put( 44, 17){\line(1,0){33}}
\put( 44, 23){\line(1,0){33}}
\put( 81, 20){\circle{10}}
\put( 86, 20){\line(1,0){29}}
\put(120, 20){\circle{10}}
\put(125, 20){\line(1,0){20}}
\put(155, 17){$\cdot$}
\put(165, 17){$\cdot$}
\put(175, 17){$\cdot$}
\put(185, 20){\line(1,0){20}}
\put(210, 20){\circle{10}}
\put( 38, 02){$b$}
\put( 78, 02){$a$}
\put(118, 02){$a$}
\put(208, 02){$a$}
\put( 38, 32){$t$}
\put( 77, 32){$s_1$}
\put(117, 32){$s_2$}
\put(200, 32){$s_{n{-}1}$}
\end{picture}
\end{center}
More precisely, there exists a basis $(T_w)_{w \in W_n}$ 
of the $A$-module $\HC_n$ such that the multiplication on $\HC_n$ 
is $A$-bilinear and completely determined by the following properties:
$$\begin{cases}
T_w T_{w'} = T_{ww'}, & \text{if $\ell(ww')=\ell(w)+\ell(w')$,} \\
(T_t-e^b)(T_t+e^{-b}) = 0, & \\
(T_{s_i}-e^a)(T_{s_i}+e^{-a})=0,& \text{if $1 \le i \le n-1$.}
\end{cases}$$

All along this paper, we also fix a field $k$ of characteristic $0$ 
and a morphism of groups $\th : \G \to k^\times$: the morphism 
$\th$ extends uniquely to a morphism of rings $\ZM[\G] \to k$ that 
we still denote by $\th$. We then set
$$k\HC_n = k \otimes_A \HC_n$$
and we still denote by $\th$ the specialization morphism $\HC_n \to k\HC_n$. 
Moreover, we assume that the following holds: there exist natural 
numbers $d$ and $e$ such that $e \ge 1$ and 
$$\left\{\begin{array}{l}
\vphantom{\DS{\frac{a}{A}}}
\textrm{\it $\th(a)^2$ is a primitive $e$-th root of unity;}\\
\th(b)^2=-\th(a)^{2d}.\vphantom{\DS{\frac{A}{a}}}
\end{array}\right.$$
If $M$ is an $\HC_n$-module, we denote by $kM$ the $k\HC_n$-module 
$k \otimes_A M$. The Grothendieck group of $k\HC_n$ will be denoted 
by $\groth(k\HC_n)$ and, if $L$ is a $k\HC_n$-module, its class in 
$\groth(k\HC_n)$ will be denoted by $[L]$. Note that 
the algebra $k\HC_n$ is split.

Since $\G$ is torsion-free, the ring $\ZM[\G]$ is integral 
and we denote by $K$ its field of fractions. Then the algebra 
$K\HC_n=K \otimes_A \HC_n$ is split semisimple. Its simple modules are 
parametrized by the set $\Bip(n)$ of bipartitions of $n$: we shall 
denote this bijection by 
$$\fonctio{\Bip(n)}{\Irr K\HC_n}{\l}{V_\l.}$$
This bijection is chosen as in \cite[10.1.2]{gecklivre}.  
We denote by $\Bip$ the set of all bipartitions 
(i.e. $\Bip=\DS{\coprod_{n \ge 0}}\Bip(n)$), the empty partition 
will be denoted by $\emptyset$ and the empty bipartition $(\emptyset,\emptyset)$ 
(which is the unique bipartition in 
$\Bip(0)$) 
will be denoted by $\emptysetb$. Since $A$ is integrally closed 
(it is a Laurent polynomial ring in several algebraically 
independent indeterminates), 
there is a well-defined decomposition map \cite[Theorem 7.4.3]{gecklivre} 
$$\db_n : \groth(K\HC_n) \longto \groth(k\HC_n).$$

\bigskip

\section{Fock space, canonical basis and Ariki's Theorem}

\medskip

The aim of this section is to recall Ariki's Theorem relating 
the canonical basis of the Fock space and the decomposition matrix 
$\db_n$. The main references are 
\cite{arikilivre}, \cite{geckLausanne} and \cite{xavierthese}.  

\bigskip

\subsection{The Fock space}\label{Fock}
Let $v$ be an indeterminate. Let $\hG$ be a free $\ZM$-module 
with basis $(h_0,\dots,h_{e-1},\dG)$ and let $(\L_0,\dots,\L_{e-1},\d)$ 
be its dual basis in $\hG^*=\Hom_\ZM(\hG,\ZM)$. 
The quantum group $\Ue$ is defined as the unital associative 
$\CM(v)$-algebra generated by elements $\{e_i, f_i~|~0 \le i \le e-1\}$ 
and $\{k_h~|~h \in \hG\}$ subject to the relations given 
for example in \cite[chapter 6]{mathaslivre}. 
We denote by $\Uepr$ the subalgebra of $\Ue$ generated by $e_i$, $f_i$, $k_{h_i}$, $k_{h_i}^{-1}$ for $i\in \{0,1,...,e-1\}$.

We fix a pair $s=(s_0,s_1) \in \ZM^2$. To $s$ is associated 
a {\it Fock space} ({\it of level $2$}) $\FG^s$: this is 
an $\Ue$-module defined as follows. 
As a $\CM(v)$-vector space, 
it has a basis given by the symbols $|\l,s\rangle$ where $\l$ 
runs over the set of bipartitions:
$$\FG^s = \mathop{\oplus}_{\l \in \Bip} 
\CM(v)\hskip0.1cm |\l,s\rangle.$$
The action of the generators $\Ue$ is given for instance in 
\cite{uglov}. By \cite{JMMO}, $\FG^s$ is an {\it integrable} 
$\Ue$-module.

If $m \in \ZM$, we denote by $\mba$ the unique element of 
$\{0,1,\dots,e-1\}$ such that $m \equiv \mba \mod e$. We then 
set 
$$\D(s)=\frac{1}{2}\sum_{j=0}^1 \frac{s_j-\sba_j}{e}(s_j+\sba_j-e).$$
Since $|{ \emptysetb},s\rangle$    is a highest
weight vector of    $\FG^s$   
and as $\FG^s$ is an integrable module, 
it follows that the submodule $\Ue  |\emptysetb,s\rangle$  
generated by  $|{\emptysetb},s\rangle$   is an irreducible module. 
We denote it by $ \mathcal{M}[s]$ and it is isomorphic to the 
irreducible $\Ue$-module with highest weight 
$-\Delta (s)\delta+\L_{\sba_0}+\L_{\sba_1}$.  
In addition, the submodule $\Uepr  |\emptysetb,s\rangle$  
generated by  $|\emptysetb,s\rangle$   is also an irreducible 
highest weight module with  weight 
$-\Delta (s)\delta+\Lambda_{\sba_0}+\Lambda_{\sba_1}$. 
We denote it by $\MC[s]'$.

\bigskip

\remark{isomorphisme} 
If $\MC$ and $\NC$ are simple $\Uepr$-modules with highest 
weight $\L$ and $\L'$ respectively, then $\MC \simeq \NC$ if and only 
if $\L \equiv \L' \mod \ZM \d$. Therefore, if $s'=(s_0',s_1') \in \ZM^2$, 
then $\MC[s]' \simeq \MC[s']'$ if and only if 
$(s_0,s_1) \equiv (s_0',s_1 ') \mod e\ZM^2$ or $(s_0,s_1) \equiv (s_1',s_0 ') \mod e\ZM^2$   .\finl

\bigskip

\subsection{Uglov's canonical basis\label{sub: uglov}} 
We shall recall here Uglov's construction of a canonical 
basis of the Fock space, which contains as a particular case 
the Kashiwara-Lusztig canonical basis of the simple $\Ue$-module 
$\MC[s]$ (which is the same as the one of  $\MC[s]'$). The reader may refer 
to Uglov's original paper \cite{uglov} for further details. 

First, recall 
that there is a unique $\CM$-linear involutive automorphism 
of algebra $\overline{~\vphantom{A}} : \Ue \longto \Ue$ 
such that, for all $i \in \{0,1,...,e-1\}$ and $h\in \hG$, we have
$$\overline{v}=v^{-1},\qquad{\overline{k_h}=k_{-h}},
\qquad{\overline{e_i}=e_i,}\qquad{\overline{f_i}=f_i}.$$
It is $\CM(v)$-antilinear (with respect to the restriction of $\overline{?}$ 
to $\CM(v)$). 

One of the main ingredients in Uglov's construction is 
a $\CM(v)$-antilinear involution 
$$\overline{~\vphantom{A}} : \FG^\sb \longto \FG^\sb$$
which is defined using the {\it wedge realization} of the Fock 
space. This involution is $\Ue$-antilinear, that is, for all 
$u \in \Ue$ and $f \in \FG^s$, we have
\equat\label{invo}
\overline{u\cdot f} = \uov \cdot \fov.
\endequat
Moreover, recall that 
\equat\label{invo vide}
\overline{|\emptysetb,s\rangle} = |\emptysetb,s\rangle.
\endequat
Now, if $\l \in \Bip$, there exists 
a unique $G(\l,s) \in \FG^s$ satisfying 
$$\left\{\begin{array}{l}
\vphantom{\DS{\frac{a}{A}}}
\overline{G(\l,s)} = G(\l,s), \\
G(\l,s) \equiv |\l,s\rangle \mod v\CM[v].
\vphantom{\DS{\frac{A}{a}}}
\end{array}\right.$$
The proof of this result is constructive and gives 
an algorithm to compute this canonical basis. 
This algorithm has been improved by Yvonne \cite{xavieralgo}.
Let us write, for $\mu \in \Bip$, 
$$G(\mu,s)=\sum_{\l \in \Bip} d_{\l,\mu}^s(v)\hskip0.1cm
|\l,s\rangle,$$
where $d_{\mu,\mu}^s(v)=1$ and $d_{\l,\mu}^s(v) \in v\CM[v]$ 
if $\l \neq \mu$. 

Finally, recall that there exists 
a unique subset $\Bip_e^s$ of $\Bip$ such that 
\equat\label{canonique}
\text{\it $\bigl(G(\l,s)\bigr)_{\l \in \Bip_e^s}$ coincides 
with Kashiwara-Lusztig canonical basis of $\MC[s]$.} 
\endequat
We set 
$$\Bip_e^s(n)=\Bip_e^s \cap \Bip(n).$$

\bigskip

\subsection{Ariki's Theorem} 
We shall recall here the statement of Ariki's Theorem, using 
the later construction of Uglov:

\bigskip

\noindent{\bf Ariki's Theorem.} 
{\it Assume that $d \equiv s_0-s_1 \mod e$. 
Then there exists a unique bijection
$$\fonctio{\Bip_e^s (n)}{\Irr k\HC_n}{\mu}{D_\mu^s}$$
such that 
$$\db_n[V_\l] = \sum_{\mu \in \Bip_e^\s(n)} d_{\l,\mu}^s(1)\hskip0.1cm 
[D_\mu^s]$$
for all $\l \in \Bip(n)$.}

\bigskip

This theorem, together with the algorithm for computing 
the Kashiwara-Lusztig canonical bases \cite{algo}, provides an efficient tool for computing 
the decomposition map $\db_n$.

\bigskip

\section{Domino insertion}

\medskip

We shall review here some combinatorial results about the {\it domino insertion 
algorithm}: the reader may refer to \cite{lam}, \cite{van} and \cite{white} 
for further details.

\bigskip
 
\subsection{Domino tableaux, bitableaux} 
Let $r\in \NM$. We denote by $\d_r$ the $2$-core $(r,r-1,\dots,1)$, 
with the convention that $\d_0=\emptyset$. Let $\PC$ denote the set of 
all partitions and let $\PC_r$ denote the set of 
partitions with $2$-core $\d_r$. We denote by 
$\PC_r(n)$ the set of partitions of $2$-weight $n$ 
and $2$-core $\d_r$, so that 
$$\PC_r = \coprod_{n \ge 0} \PC_r(n)$$
$$\PC=\coprod_{r \ge 0} \PC_r.\leqno{\text{and}}$$
Note that partitions in $\PC_r(n)$ are partitions of 
$2n+\DS{\frac{r(r-1)}{2}}$. 
Let $\qb : \PC \to \Bip$ 
the map sending a partition to its $2$-quotient. By composition
 with the bijection 
$$ \begin{array}{ccc} 
  \Bip &\longrightarrow& \Bip \\
 (\lambda^0,\lambda^1) & \mapsto & \left\{ \begin{array}{ll} 
(\lambda^0,\lambda^1) & \textrm{ if } r\textrm{ is even}\\
(\lambda^1,\lambda^0)& \textrm{ if } r\textrm{ is odd}
 \end{array}\right.
   \end{array}
   $$
the restriction of $\qb$ 
to $\PC_r$ induces a bijection 
$$\qb_r : \PC_r \stackrel{\sim}{\longrightarrow} \Bip.$$
The restriction of $\qb_r$ to $\PC_r(n)$ induces 
a bijection
$$\PC_r(n) \stackrel{\sim}{\longrightarrow} \Bip(n)$$
that we still denote by $\qb_r$ if no confusion may arise.

Let $\SDT_r(n)$ denote the set of standard domino tableaux 
whose shape lies in $\PC_r(n)$ (and filled with dominoes 
with entries $1$, $2$,\dots, $n$) and let $\SBT(n)$ 
denote the set of standard bitableaux of total size $n$ 
(filled again with boxes with entries $1$, $2$,\dots, $n$). 
We denote by $\qbt_r : \SDT_r(n) \to \SBT(n)$ the bijection 
obtained as a particular case of 
\cite[Theorem 7.3]{CL}. If $\lamb : \SBT(n) \to \Bip(n)$ 
(respectively $\Delb$) sends a bitableau (respectively a domino tableau) 
to its shape, then the diagram
\equat\label{domino n}
\diagram
\SDT_r(n) \rrto^{\DS{\qbt_r}}_{\DS{\sim}}
\ddto_{\DS{\Delb}} && \SBT(n) \ddto^{\DS{\lamb}} \\
&&\\
\PC_r(n) \rrto^{\DS{\qb_r}}_{\DS{\sim}} && \Bip(n)
\enddiagram
\endequat
is commutative.

\bigskip

\subsection{Orders between bipartitions\label{sub: order}} 
The bijection $\qb_r$ allows us to define several partial orders 
on the set of bipartitions. First, let $\trianglelefteq$ denote 
the dominance order on $\PC$. We then define the partial order 
$\dominance{r}$ on $\Bip$ as follows: if $\l$, $\mu \in \Bip$, 
then 
$$\l \dominance{r} \mu \Longleftrightarrow \qb_r^{-1}(\l) \trianglelefteq 
\qb_r^{-1}(\mu).$$

\bigskip

\remark{asymptotique}
If $r\geq n-1$ and $s\geq n-1$, then 
the partial orders $\dominance{r}$ and $\dominance{s}$ coincide 
on $\Bip(n)$. In fact, they coincide with the classical 
{dominance order} on bipartitions that we shall denote 
by $\dominance{\infty}$.\finl

\bigskip

\example{n=2} 
Here are the orders $\dominance{0}$ and 
$\dominance{1}=\dominance{2}=\dots = \dominance{\infty}$ on $\Bip(2)$:
$$(\emptyset;11)~\dominance{0}~(11;\emptyset) ~\dominance{0}~(1;1) ~\dominance{0}~
(\emptyset;2) ~\dominance{0}~(2;\emptyset)$$
$$(\emptyset;11)~\dominance{\infty}~(\emptyset;2) ~\dominance{\infty}~(1;1) ~\dominance{\infty}~(11;\emptyset) ~\dominance{\infty}~(2;\emptyset)$$
\bigskip

\example{n=3} 
Here are the orders $\dominance{0}$, $\dominance{1}$ and 
$\dominance{2}=\dominance{3}=\dots = \dominance{\infty}$ on $\Bip(3)$.
In the diagram, we put an arrow between two bipartitions $\l \longto \mu$ 
if $\mu \dominance{i} \l$ and if there is no other bipartition $\nu$ 
such that $\mu \dominance{i} \nu \dominance{i}\l$.

{\small 
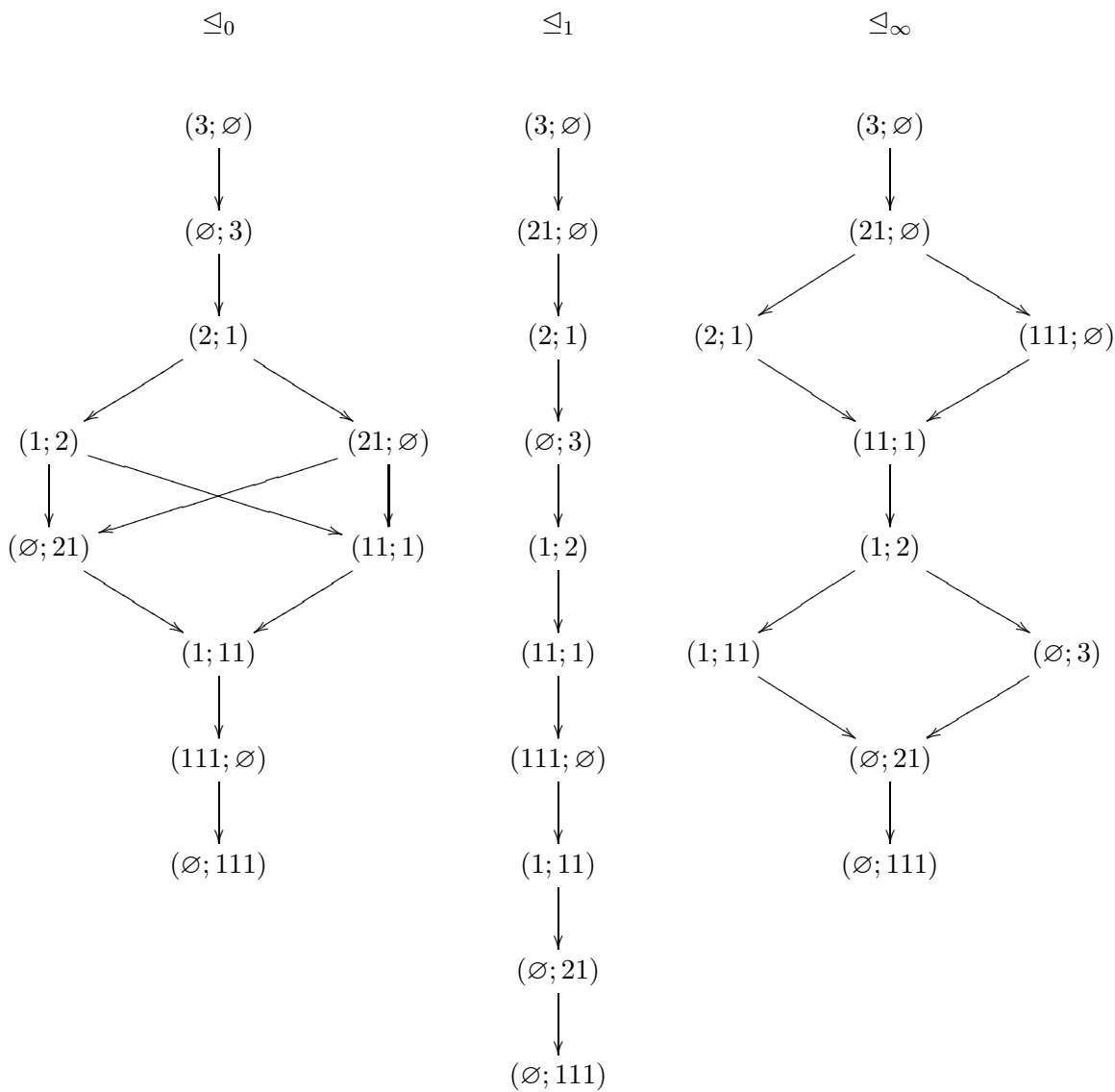
$$\diagram
& \dominance{0} && \dominance{1} && 
\dominance{\infty} & \\
& (3;\vide) \dto &  & (3;\vide) \dto && (3;\vide) \dto & \\
& (\vide;3) \dto &  & (21;\vide) \dto && (21;\vide) \dlto\drto & \\
& (2;1) \drto\dlto & & (2;1) \dto & (2;1) \drto && (111;\vide) \dlto \\
(1;2) \dto \drrto && (21;\vide) \dto \dllto & (\vide;3) \dto &&  (11;1) \dto & \\
(\vide;21) \drto && (11;1) \dlto & (1;2) \dto & & (1;2) \dlto\drto & \\
& (1;11) \dto & & (11;1) \dto & (1;11) \drto && (\vide;3) \dlto \\
& (111;\vide) \dto & & (111;\vide) \dto & & (\vide;21) \dto & \\
& (\vide;111) & & (1;11) \dto & & (\vide;111) & \\ 
&&& (\vide;21) \dto& \\
&&& (\vide;111) & 
\enddiagram$$}

\bigskip 

\subsection{Domino insertion algorithm} 
Recall that $r \in \NM$ is fixed. 
If $w \in W_n$, then the {\it domino insertion algorithm}, as 
described for instance in \cite{lam}, associates 
to $w$ a standard domino tableau $P_r(w) \in \SDT_r(n)$. 
We set
$$S_r(w)=\qbt_r(P_r(w)).$$
Then $S_r(w)$ is a standard bitableau of total size $n$. 
We also set 
$$Q_r(w)=P_r(w^{-1})\quad\text{and}\quad T_r(w)=S_r(w^{-1})=\qbt_r(Q_r(w)).$$
It turns out that $P_r(w)$ and $Q_r(w)$ have the same shape, 
as well as $S_r(w)$ and $T_r(w)$. Also, if we set 
$$\SBT^{(2)}(n)=\{(S,T) \in \SBT(n) \times \SBT(n)~|~\lamb(S)=\lamb(T)\},$$
then the map
\equat\label{bij}
\fonctio{W_n}{\SBT^{(2)}(n)}{w}{(S_r(w),T_r(w))}
\endequat
is bijective. If $w \in W_n$, then we set
$$\lamb_r(w)=\lamb(S_r(w)) \in \Bip(n).$$
Note that
\equat\label{forme stable}
\lamb_r(w) = \lamb(T_r(w)) = \lamb_r(w^{-1}).
\endequat

\bigskip

\remark{asymptotique bij}
If $r$, $r' \ge n-1$, then the bijection $W_n \to \SBT^{(2)}(n)$, 
$w \mapsto (S_{r'}(w),T_{r'}(w))$ coincides with the bijection 
\ref{bij}. We shall denote it by $W_n \to \SBT^{(2)}(n)$, 
$w \mapsto (S_\infty(w),T_\infty(w))$. Similarly, we shall 
set $\lamb_\infty(w)=\lamb(S_\infty(w))$.\finl

\bigskip

\section{Kazhdan-Lusztig theory}

\medskip

The aim of this section is to show that Kazhdan-Lusztig theory 
{\it with unequal parameters} should 
provide cellular data on $\HC_n$ which are compatible with 
Ariki's Theorem. The reader may refer to \cite{lusztig} for 
the foundations of this theory. 

\medskip

We denote by $\HC_n \to \HC_n$, $h \mapsto \hov$ the unique 
automorphism of ring such that $\overline{e^\g}=e^{-\g}$ and 
$\overline{T}_w=T_{w^{-1}}^{-1}$ for all $\g \in \G$ and $w \in W_n$. 
It is an $A$-antilinear 
involutive automorphism of the ring $\HC_n$. 

\bigskip

\subsection{Kazhdan-Lusztig basis} 
From now on, we fix a total order $\le$ on $\G$ which endows 
$\G$ with a structure of ordered group. 
$$A_{\leqslant 0}=\mathop{\oplus}_{\g \le 0} \ZM e^\g
\quad\text{and}\quad A_{< 0}=\mathop{\oplus}_{\g < 0} \ZM e^\g.$$
Then $A_{\leqslant 0}$ is a subring of $A$ and $A_{<0}$ is an ideal 
of $A_{\leqslant 0}$. We also set
$$\HC_n^{<0}=\mathop{\oplus}_{w \in W_n} A_{<0} T_w.$$
It is a sub-$A_{\leqslant 0}$-module of $\HC_n$. 

If $w \in W_n$, 
then \cite[Theorem 5.2]{lusztig} 
there exists a unique element $C_w^\ordre \in \HC_n$ 
such that 
$$\left\{\begin{array}{l}
\vphantom{\DS{\frac{a}{A}}}
\overline{C_w^\ordre} = C_w^\ordre; \\
C_w^\ordre \equiv T_w \mod \HC_n^{<0}.
\vphantom{\DS{\frac{A}{a}}}
\end{array}\right.$$
Also, $(C_w^\ordre)_{w \in W_n}$ is an $A$-basis of $\HC_n$ and this basis depends on the choice of the order $\ordre$ on $\G$. We define 
the preorders $\pre{\LC,\ordre}$, $\pre{\RC,\ordre}$ and 
$\pre{\LC\RC,\ordre}$ on $W_n$ 
as in \cite[\S 8.1]{lusztig} and we denote by 
$\cell{\LC,\ordre}$, $\cell{\RC,\ordre}$ and $\cell{\LC\RC,\ordre}$ the equivalence 
relations respectively associated to these preorders. 
Here, the exponents or the 
indices $\ordre$ stands for emphasizing the fact that the objects really 
depends on the order $\le$  on $\G$.

\bigskip

\subsection{Bonnaf\'e-Geck-Iancu-Lam conjectures}\label{Nbgil}
From now on, we shall assume that the parameters $a$, $b$ in $\G$ 
are strictly positive and that $b \not\in \{a,2a,\dots, (n-1)a\}$.

\begin{quotation} 
\noindent{\it If $b < (n-1) a$, then we denote by $r$ the unique 
natural number such that $ra < b < (r+1)a$. If $b > (n-1)a$, then 
we may choose for $r$ any value in $\{n-1,n,n+1,\dots\} \cup \{\infty\}$.}
\end{quotation}

Note that the flexibility on the choice of $r$ whenever 
$b > (n-1) a$ will not change the statements below 
(see Remarks \ref{asymptotique} and \ref{asymptotique bij}). 
Note also that, once $a$ and $b$ are fixed elements of $\G$, then 
the number $r$ depends only on the choice of the order $\le$ on $\G$ 
(again with some flexibility if $b > (n-1)a$). 

\bigskip

\begin{quotation}
\noindent{\bf Conjecture A.} 
{\it With the above definition of $r$, we have, for all $w$, $w' \in W_n$:
\begin{itemize}
\itemth{a} $w \cell{\LC,\ordre} w'$ if and only if $T_r(w)=T_r(w')$.

\itemth{b} $w \cell{\RC,\ordre} w'$ if and only if $S_r(w)=S_r(w')$.

\itemth{c} $w \cell{\LC\RC,\ordre} w'$ if and only if 
$\lamb_r(w)  = \lamb_r(w')$.

\itemth{c^+} $w \pre{\LC\RC,\ordre} w'$ if and only if 
$\lamb_r(w)  \dominance{r} \lamb_r(w')$.
\end{itemize}}
\end{quotation}

\bigskip

\remark{blabla}
First, note that the statements (a) and (b) in Conjecture A are 
equivalent. Note also that, if Lusztig's Conjectures {\bf P1-P15} in 
\cite[Chapter 14]{lusztig} hold, then (a) and (b) imply (c). 

Note also that the statements (a), (b) and (c) have been 
proved whenever $b > (n-1)a$ ({\it asymptotic case}: see \cite[Theorem 7.7]{iancu} 
and \cite[3.9]{bilatere}): the statement ($\crm^+$) has been proved only 
if $w$ and $w'$ have the same {\it $t$-length} (i.e. if the numbers 
of occurrences of $t$ in a reduced expression of $w$ and $w'$ are equal), 
see \cite[3.8]{bilatere}. 

Statements (a), (b) and ($\crm^+$) 
have also been proved whenever $a=2b$ or $3a=2b$ (see \cite[Theorem 3.14]{BGIL}).

Finally, we must add that Conjecture $A$ is ``highly probable'': 
it has been checked for $n \le 6$ and is compatible with many 
other properties of Kazhdan-Lusztig cells and other conjectures 
of Lusztig. For a detailed discussion about this, the reader may refer 
to \cite{BGIL}.\finl

\bigskip

{\it From now on, and until the end of this paper, we assume that 
Conjecture A holds.}
We shall now define a basis of $\HC_n$ which depends on $a$, $b$ and $\le$ 
as follows. If $(S,T) \in \SBT^{(2)}(n)$, let $w$ be the unique 
element of $W_n$ such that $S=S_r(w)$ and $T=T_r(w)$: we then 
set
$$C_{S,T}^\ordre = (C_w^\ordre)^\dagger,$$
where $\dagger : \HC_n \to \HC_n$ is the unique $A$-algebra involution 
such that $T_s^\dagger = - T_s^{-1}$ for all $s \in S_n$. 
Now, if $\l \in \Bip(n)$, we denote by $\SBT(\l)$ the set of standard 
bitableaux of shape $\l$ (i.e. $\SBT(\l)=\lamb^{-1}(\l)$). Finally, 
we denote by $\HC_n \to \HC_n$, $h \mapsto h^*$ the unique 
$A$-linear map such that 
$$T_w^*=T_{w^{-1}}$$
for all $w \in W_n$. It is an involutive anti-automorphism 
of the $A$-algebra $\HC_n$. 
Then we have constructed 
a quadruple $((\Bip(n),\dominance{r}),\SBT,C^\ordre,*)$ where
\begin{itemize}
\item[$\bullet$] $(\Bip(n),\dominance{r})$ is a poset;

\item[$\bullet$] For each $\l \in \Bip(n)$, $\SBT(\l)$ is a finite 
set;

\item[$\bullet$] Let $\SBT^{(2)}(n)=\DS{\coprod_{\l \in \Bip(n)}} \SBT(\l) \times \SBT(\l)$ then the map 
$$\fonction{C^\ordre}{\SBT^{(2)}(n)}{\HC_n}{(S,T)}{C_{S,T}^\ordre}$$
is injective and its image is an $A$-basis 
$(C_{S,T}^\ordre)_{(S,T)\in \SBT^{(2)}(n)}$ of $\HC_n$;

\item[$\bullet$] The map $* : \HC_n \to \HC_n$ 
is an $A$-linear involutive anti-automorphism of the ring $\HC_n$ such that 
$(C_{S,T}^\ordre)^*=C_{T,S}^\ordre$ for all $(S,T) \in \SBT^{(2)}(n)$.
\end{itemize}

\bigskip

The next conjecture is taken from \cite[Conjecture C]{BGIL}.

\bigskip

\begin{quotation}
\noindent{\bf Conjecture B.} 
{\it The quadruple $((\Bip(n),\dominance{r}),\SBT,C^\ordre,*)$ is a cell 
datum on $\HC_n$, in the sense of Graham-Lehrer \cite{GL}.}
\end{quotation}

\bigskip

\remark{rem cell}
For a discussion of some evidences for this conjecture, 
see again \cite{BGIL}. Note that if Lusztig's Conjectures 
{\bf P1-P15} in \cite[Chapter 14]{BGIL} hold and if Conjecture A above 
holds, then Conjecture B holds by a recent work of Geck 
\cite{geckcell}.\finl

\bigskip

In the rest of this paper, we shall give some more evidences for 
Conjectures A and B, towards their compatibilities with Ariki's Theorem. 

\bigskip

\begin{quotation}
\noindent{\bf Hypothesis:} 
{\it From now on, and until the end of this paper, we assume that 
Conjectures A and B hold. We shall denote by $\CC^\ordre$ the cell 
datum $((\Bip(n),\dominance{r}),\SBT,C^\ordre,*)$.}
\end{quotation}

\bigskip

\subsection{Interpretation of Ariki's Theorem} 
Since we assume that Conjectures A and B hold, the general 
theory of cellular algebras \cite{GL} shows that the 
cell datum $\CC^\ordre$ provides us with a family 
of Specht modules $S_\l^\ordre$ associated to each bipartition $\l$ of $n$. 
Moreover, $S_\l^\ordre$ is endowed with a bilinear 
form $\phi_\l^\ordre : S_\l^\ordre \times S_\l^\ordre \to A$ such that, 
for all $h \in \HC_n$ 
and $x$, $y \in S_\l^\ordre$, we have
\equat\label{h invariante}
\phi_\l^\ordre(h\cdot x,y)=\phi_\l^\ordre(x,h^* \cdot y).
\endequat
We denote by $k\phi_\l^\ordre$ the specialization of $\phi_\l^\ordre$ 
to $kS_\l^\ordre$ (through $\th : \G \to k^\times$). We then set
$$D_\l^\ordre=kS_\l^\ordre/\rad(k\phi_\l^\ordre).$$
By \ref{h invariante},  $\rad(k\phi_\l^\ordre)$ is a $k\HC_n$-submodule 
of $S_\l^\ordre$, so that $D_\l^\ordre$ is a $k\HC_n$-module. We set 
$$\Bip_{(k)}^\ordre(n)=\{\l \in \Bip(n)~|~D_\l^\ordre \neq 0\}.$$
Then, by the theory of cellular algebras \cite[Thm. 3.4]{GL}, 
$(D_\l^\ordre)_{\l \in \Bip_{(k)}^\ordre(n)}$ is a family of representatives 
of isomorphism classes of irreducible $k\HC_n$-modules.

\bigskip

\remark{dependance}
Note that the $k\HC_n$-modules $D_\l^\ordre$, as well as the 
$\HC_n$-modules $S_\l^\ordre$, depend heavily on the choice 
of the order $\le$ on $\G$ (they depend on $r$): several choices for $\le$ 
will lead to non-isomorphic $k\HC_n$-modules. In particular, the set 
 $\Bip_{(k)}^\ordre(n)$ does depend on $\le$ (see \S \ref{gen}). \finl
\bigskip

Once we accept Conjectures A and B, it is highly probable that the Specht modules satisfy the following property.
\begin{quotation}
\noindent{\bf Conjecture B'.} 
{\it For all $n\in \mathbb{N}$ and for all $\l\in \Bip(n)$, we have
$$KS^{\ordre}_{\l}\simeq V_{\l}$$}
\end{quotation}

\bigskip

Recall that $\th$ is the specialization morphism $\HC_n \to k\HC_n$ such that 
there exist natural 
numbers $d$ and $e$ such that $e \ge 1$ and
$$\left\{\begin{array}{l}
\vphantom{\DS{\frac{a}{A}}}
\textrm{\it $\th(a)^2$ is a primitive $e$-th root of unity;}\\
\th(b)^2=-\th(a)^{2d}.\vphantom{\DS{\frac{A}{a}}}
\end{array}\right.$$

The following result is a generalization of \cite{arikillt} and \cite{gj} (see also \cite{geckcell}).

\begin{theo}\label{main theorem}
Assume that Conjectures A, B and B' hold. Let $\l \in \Bip(n)$.  
Let $p$ be the unique integer such that 
$$d+pe \le r < d+(p+1) e.$$
Put 
$$s=(d+p e,0)$$
then $\Bip_{(k)}^\ordre(n)=\Bip_e^s(n)$ and 
$$[kS_\l^\ordre:D_\mu^\ordre]=d_{\l,\mu}^s(1)$$
for all $\mu \in \Bip_{(k)}^\ordre(n)$. In other words, 
$D_\mu^\ordre\simeq D_\mu^s$ and 
$$\db_n [V_\l] = \sum_{\mu \Bip_{(k)}^\ordre(n)} d_{\l,\mu}^s(1) [D_\mu^\ordre].$$
\end{theo}

\begin{proof}
Assume that $p$ is the 
unique integer such that 
$$d+pe\le r < d+(p+1) e.$$
and put 
$$s=(d+p e,0)$$
Let $\mu\in  \Bip_e^s(n)$. Then, because of the above condition on $p$, 
we can use the same strategy as in the proof
 of \cite[Theorem 5.4]{gj} to show that:
$$G(\mu,s)=  |\mu,s\rangle +  \sum_{\l \dominance{r} \mu,\ \l\neq \mu } d_{\l,\mu}^s(v)\hskip0.1cm
|\l,s\rangle,$$
Now, assume that $\nu\in  \Bip_{(k)}^\ordre  (n)$. Then by Ariki's Theorem there exists $\mu\in  \Bip_e^s(n)$ such that
 for all $\lambda\in \Bip(n)$:
$$\db_{\lambda,\nu}=d^s_{\lambda,\mu}(1),$$
where $\db_{\l,\nu}=[kS_\l^\ordre:D_\nu^\ordre]$. 
We have :
$$\db_{\mu,\nu}=1\textrm{ and }\db_{\lambda,\nu}=0\textrm{ if }\mu \dominance{r} \lambda\textrm{ and }\mu\neq \l$$
 By the property of cellular algebras \cite[Proposition 3.6]{GL}, 
this implies that $\nu=\mu$. As we have a bijection between 
 $ \Bip_{(k)}^\ordre  (n)$ and $\Bip_e^s(n)$, 
this implies that these two sets are equal. 
The rest of the theorem is now obvious.
\end{proof}

\bigskip

\remark{ariki remark} 
The above theorem gives a conjectural interpretation of Ariki's parametrization 
of simple $k\HC_n$-modules in terms of a new cellular structure on 
$\HC_n$ (coming from Kazhdan-Lusztig's theory).\finl

\bigskip

\subsection{Uglov bipartitions}\label{ug}

The bipartitions in $\Bip_e^s(n)$ are known as Uglov bipartitions. They are constructed by using the crystal graph of the associated Fock space representation (see for example \cite{arikilivre} for details on crystal graphs). 
 In general, we only have a recursive definition for them. However, 
there exist non-recursive characterizations of such bipartitions 
in particular cases:
\begin{itemize}
 \item in the case where $-e\leq d+pe <0$, 
see \cite{FLOTW} and \cite[Proposition 3.1]{Jbijection} 
(the Uglov bipartitions are then called the FLOTW bipartitions);

\item in the case where $d+p e >n-1-e$, see \cite{aj}, \cite{akt} (the Uglov bipartitions are then called the Kleshchev bipartitions). Note that this includes  the ``asymptotic case'', that is, the case where $b>(n-1)a$.

\end{itemize}

\bigskip

\subsection{Jantzen filtration\label{sub: jantzen}} 
Once we believe in Conjectures A and B, it is natural to try 
to develop the theory as in the type A case. For instance, 
one can define a Jantzen filtration on $kS_\l^\ordre$ as follows. 
First, there exists a discrete valuation ring $\OC \subset K$ 
containing $A$ such that, if we denote by $\pG$ the maximal 
ideal of $\OC$, then $\pG \cap A = \Ker \th$. Since $K$ 
is the field of fractions of $A$, the map $\th : A \to k$ extends 
to a map $\thet : \OC \to k$ with kernel $\pG$ and we have 
$$k\HC_n = k \otimes_\OC \OC\HC_n$$
where $\OC\HC_n = \OC \otimes_A \HC_n$. Similarly, if $\l \in \Bip(n)$, 
then $\OC S_\l^\ordre$ is 
an $\OC\HC_n$-module and the extension of scalars defines a bilinear 
form $\OC\phi_\l^\ordre$ on $\OC S_\l^\ordre$. We then set, for 
all $i \ge 0$, 
$$\OC S_\l^\ordre(i)=\{x \in \OC S_\l^\ordre~|~\forall~y \in 
\OC S_\l^\ordre,~\OC\phi^\ordre_\l(x,y) \in \pG^m\}.$$
We then set 
$$kS_\l^\ordre(i) = \bigl(\OC S_\l^\ordre(i) + \pG S_\l^\ordre)/\pG S_\l^\ordre.$$
Then there exists $m_0$ such that $kS_\l^\ordre(m_0)=0$ and the 
$S_\l^\ordre(i)$'s are $k\HC_n$-submodules of $kS_\l^\ordre$. 
Moreover, $kS_\l^\ordre(0)=kS_\l^\ordre$ and 
$kS_\l^\ordre(1)=\rad k\phi_\l^\ordre$. 
The {\it Jantzen filtration} of $kS_\l^\ordre$ is the filtration  
$$0=kS_\l^\ordre(m_0) \subseteq kS_\l^\ordre(m_0-1) \subseteq \cdots 
\subseteq kS_\l^\ordre(1) \subseteq kS_\l^\ordre(0)=kS_\l^\ordre.$$
The next conjecture proposes an interpretation of the polynomials 
$d_{\l,\mu}^s$ as a {\it $v$-decomposition matrix}. This is a generalization of a conjecture by Lascoux, Leclerc and Thibon \cite[\S 9]{llt}.

\bigskip

\begin{quotation}
\noindent{\bf Conjecture C.} 
{\it Let $\l \in \Bip(n)$ and $\mu \in \Bip_{(k)}^\ordre(n)$. 
Let $p$ be the unique integer such that 
$$(d+pe)a< b < (d+(p+1)e)a.$$
Put 
$$s=(d+p e,0),$$ then
$$d_{\l,\mu}^s(v)=\sum_{i \ge 0} ~[kS_\l^\ordre(i)/kS_\l^\ordre(i+1): D_\mu^\ordre]~v^i.$$}
\end{quotation}

\bigskip

It would also be very interesting to find an analogue of 
{\it Jantzen's sum formula}.  This formula could be obtained using the matrix described by Yvonne in \cite[\S 7.4]{xavierthese}.

\bigskip

\subsection{An example: the asymptotic case} 
{\it Assume here, and only in this subsection, that $b > (n-1) a$ (in other words, if $r \ge n-1$)}. Then Theorem \ref{main theorem} 
holds without assuming that Conjectures A and B are true in this case. 
However, note that, at the time this paper is written, the Conjecture 
A is not fully proved. As it is explained in Remark \ref{blabla}, only 
(a), (b) and (c) are proved and part of ($\crm^+$). However, we can define 
an order $\dominancepr$ on $\Bip(n)$ as follows: if $\l$, $\mu \in \Bip(n)$, 
we write $\l \dominancepr \mu$ if $w \pre{\LC\RC,\ordre} w'$ for 
some (or all) $w$ and $w'$ in $W_n$ such that $\lamb_\infty(w)=\l$ and $\lamb_\infty(w')=\mu$. Then, by the work of Geck-Iancu 
\cite{geck-iancu} and Geck \cite{geckcell}, we have 
that $((\Bip(n),\dominancepr),\SBT,C^\ordre,*)$ is a cell datum on 
$\HC_n$ and it is easily checked that the proof of Theorem 
\ref{main theorem} remains valid if we replace everywhere 
$\dominance{r}$ by $\dominancepr$. 

\medskip

Note also that the cell datum $((\Bip(n),\dominancepr),\SBT,C^\ordre,*)$ 
is roughly speaking equivalent to the cell data constructed by 
Graham and Lehrer \cite[Theorem 5.5]{GL} or Dipper, James and Mathas \cite[Theorem 3.26]{DJMa}. 

\bigskip

\section{Generic Hecke algebra}\label{gen}

In the last section, we have shown that, if we assume that 
conjectures $A$ and $B$ hold,  the choice of a pair 
$(a,b) \in \G^2$ and the choice of a total order on $\G$ lead  
to  an ``appropriate'' representation theory for the associated 
Hecke algebra of type $B_n$. In this section, we show that these 
results can be applied to find several ``Specht modules'' theories 
for the {\it same} algebra: we shall work with the group $\G=\ZM^2$, 
with $a=(1,0)$ and $b=(0,1)$ and the main theme of this section will be 
to make the order $\leqslant$ vary.

\medskip

\begin{quotation}
\noindent{\bf Hypotheses and notation:} {\it From now on, we assume that 
$\G=\ZM^2$, $a=(1,0)$ and $b=(0,1)$. We set $Q=e^b$ and $q=e^a$ so that 
$A=\ZM[Q,Q^{-1},q,q^{-1}]$.}
\end{quotation}

In this case, $\HC_n$ is called the {\it generic Hecke algebra} of type $B_n$.
Recall that 
$$(T_t-Q)(T_t+Q^{-1})=(T_{s_i}-q)(T_{s_i}+q^{-1})=0$$
for all $i \in \{1,2,\dots,n-1\}$. Now, any choice 
of elements $Q_0$ and $q_0$ in $k^\times$ defines a unique 
morphism $\th : \G \to k^\times$, $a \mapsto v_0$, $b \mapsto V_0$. 

\medskip

\subsection{Orders on $\mathbb{Z}^2$.} 
Our main theme in this section is to use the possibility of 
endowing $\G=\ZM^2$ with several orders. 
For instance, if $\xi \in \RM_{>0} \setminus \QM$ 
(a positive irrational number), then we can define
$$\text{\it $(m,n) \pre{\xi} (m',n')$ if and only if 
$m+\xi n \le m' + \xi n'$.}$$
This defines a total order on $\ZM^2$. Moreover, $0 <_\xi a$ and $0 <_\xi b$, 
so that all the results of this paper can be applied. If $r$ denotes 
the entire part $[\xi]$ of $\xi$, then
$$ra < b < (r+1) a.$$
For simplification, we shall denote by $S_\l^\xi$ the $\HC_n$-module 
$S_\l^{\pre{\xi}}$ and $D_\l^\xi$ the $k\HC_n$-module $D_\l^{\pre{\xi}}$. 
We also set $\Bip_{(k)}^\xi(n)=\Bip_{(k)}^{\pre{\xi}}(n)$.   We also  assume that the following holds: there exist natural 
numbers $d$ and $e$ such that $e \ge 1$ and 
$$\left\{\begin{array}{l}
\vphantom{\DS{\frac{a}{A}}}
\textrm{\it $\th(a)^2$ is a primitive $e$-th root of unity;}\\
\th(b)^2=-\th(a)^{2d}.\vphantom{\DS{\frac{A}{a}}}
\end{array}\right.$$

We can now restate the Theorem \ref{main theorem} in this case:

\bigskip

\begin{theo}\label{generique}
Assume that Conjectures A, B and B' hold. 
Let $\l\in \Bip(n)$. Let $p$ be the unique integer such that 
$$d+pe < \xi <  d+(p+1) e.$$
Put 
$$s=(d+p e,0)$$
then $\Bip_{(k)}^\xi(n)=\Bip_e^s(n)$ and 
$$[kS_\l^\xi:D_\mu^\xi]=d_{\l,\mu}^s(1)$$
for all $\mu \in \Bip_{(k)}^\xi(n)$. In other words, 
$D_\mu^\xi\simeq D_\mu^s$ and 
$$\db_n [V_\l] = \sum_{\mu \in\Bip_{(k)}^\xi(n)} d_{\l,\mu}^s(1) [D_\mu^\xi].$$
\end{theo}

\bigskip

This Theorem shows that, if Conjectures A and B hold, then 
the different choices of $\xi$ (or of other orders on $\ZM^2$) lead 
to:
\begin{itemize}
\item[$\bullet$] Different cellular structures on $\HC_n$;

\item[$\bullet$] Different families of Specht modules;

\item[$\bullet$] Different parametrizations of the simple $k\HC_n$-modules,
\item[$\bullet$] Different ``$v$-decomposition matrices'' as defined in \ref{sub: jantzen} (despite the fact that the decomposition matrices must be equal up to permutations of the rows and the columns).
\end{itemize}
For instance, if $\xi$ and $\xi'$ are two positive irrational numbers, 
then the $K\HC_n$-modules $KS_\l^\xi$ 
and $KS_\l^{\xi'}$ are isomorphic, but the 
$\HC_n$-modules $S_\l^\xi$ and $S_\l^{\xi'}$ might be non-isomorphic. 

\bigskip

\remark{asymptotic specht}
If $\xi > n-1$, then Geck, Iancu and Pallikaros \cite{GIP} have shown 
that the Specht modules $S_\l^\xi$ are isomorphic to the ones 
constructed by Dipper, James and Murphy \cite{DJMU}.\finl

\subsection{Crystal isomorphisms}
Let $\xi_1$ and $\xi_2$ be irrational numbers. 
Let $r_1 $  and $r_2$  be the natural numbers such that 
$$r_1  < \xi_1 < (r_1+1)\ \ \textrm{ and }\ \ r_2  < \xi_2 < (r_2+1).$$
 Then 
 if we use the order $\pre{\xi_1}$, we obtain a 
 complete set of non-isomorphic simple modules for the specialized algebra $k\HC_n$:
$$\{D_{\lambda}^{\xi_1}\ |\ \lambda\in \Bip^{\xi_1}_{(k)}\}.$$
 On the other hand, if we use the order $\pre{\xi_2}$,  we obtain a 
 complete set of non-isomorphic simple modules for the same algebra:
$$\{D_{\lambda}^{\xi_2}\ |\ \lambda\in \Bip^{\xi_2}_{(k)}\}.$$
By Theorem \ref{generique}, there exist $s^1\in \mathbb{Z}^2$ and $s^2\in \mathbb{Z}^2$ such that:
 $$\Bip_{(k)}^{\xi_1}(n)=\Bip_e^{s^1}(n)\ \ \textrm{ and }\ \ \Bip_{(k)}^{\xi_2}(n)=\Bip_e^{s^2}(n)$$
and we have a bijection
$$c:\Bip_e^{s^1}(n)\to \Bip_e^{s^2}(n)$$
which is uniquely determined by the condition that:
$$\textrm{For all }\mu\in\Bip_e^{s^1}(n),\ D_{\mu}^{\xi_1}\simeq D_{c(\mu)}^{\xi_2}.$$
In this section, we want to explicitly determine the bijection $c$. To do this, we first note that this map induces a bijection between the vertices of the crystal graphs of $\MC[s^1]$ and $\MC[s^2]$. 

The modules $\MC[s^1]'$ and $\MC[s^2]'$ are isomorphic as $\Uepr$ modules. As a consequence, the crystal graphs of $\MC[s^1]$ and $\MC[s^2]$ are also isomorphic. This bijection may be obtained by following a sequence of arrows back to the
 empty bipartition in the crystal graph of   $\MC[s^1]$  and then applying
 the reversed sequence from the empty bipartition of   $\MC[s^2]$. One can also define it as follows. Let 
$$\mathcal{B}_{s^1}:=\{ G(\mu,s^1)\ |\ \mu\in   \Bip_e^{s^1}\}$$
be the elements of the canonical basis of $\MC[s^1]$ and let 
$$\mathcal{B}_{s^2}:=\{ G(\mu,s^2)\ |\ \mu\in   \Bip_e^{s^2} \}$$
be the elements of the canonical basis of $\MC[s^1]$. Then if we specialize the element of $\mathcal{B}_{s^1}$ or 
$\mathcal{B}_{s^2}$ to $v=1$, we obtain a basis of the same irreducible highest weight $\mathcal{U}(\widehat{\mathfrak{sl}_e})$-modules. Hence, for all $\mu\in\Bip_e^{s^1}(n)$, there exists 
$\gamma (\mu)\in \Bip_e^{s^2}(n)$ such that 
$$\textrm{for all $\lambda\in \Bip (n)$, }d^{s^1}_{\lambda,\mu} (1)=d^{s^2}_{\lambda,\gamma (\mu)}(1).$$
Then we have a bijection
$$\gamma:\Bip_e^{s^1}(n)\to \Bip_e^{s^2}(n).$$
This bijection has been explicitly described, in a non-recursive way in \cite{Jbijection} (see also the generalization in \cite{JLbijection}). 
\begin{theo}\label{bijection}
Assume that Conjectures $A$ and $B$ hold. Then for all $\mu\in\Bip_e^{s^1}(n)$, we have
$$c (\mu)=\gamma (\mu).$$
\end{theo}

\begin{proof}
Let $r_1$ (respectively $r_2$)  denote the entire part of $\xi_1$ (respectively $\xi_2$).  Let $\mu\in\Bip_e^{s^1}(n)$. By Ariki's Theorem and the theory of cellular algebras, $c(\mu)$ is the maximal element  with respect to $\unlhd_{r_2}$ such that  $|c(\mu),s^1 \rangle$ appears with non-zero coefficient  in $G(\mu,s^1)$. 

We have $[\CM S_\l^\xi:D_\mu^{\xi_1}]=[\CM S_\l^\xi:D_{c(\mu)}^{\xi_2}]$ and
   $[\CM S_\l^\xi:D_\mu^{\xi_1}]= d_{\lambda,\mu}^{s^1}(1)=   d_{\lambda,\gamma(\mu)}^{s^2}(1)= [\CM S_\l^\xi:D_{c(\mu)}^{\xi_2}]$. Hence $\gamma (\mu)$ is the maximal element with respect to  $\unlhd_{r_2}$ such that
$$d^{s^1}_{\gamma (\mu),\mu} (1)\neq 0$$
This implies that
$$c (\mu)=\gamma (\mu).$$
This concludes the proof.
\end{proof}

\remark{rem: crystal}
Note that using the non-recursive parametrization of  $\Bip_e^{s}(n)$ known in particular cases (see \S \ref{ug}), it is possible to obtain nice characterizations of  $\Bip_e^{s}(n)$ in all cases by applying the above bijection.

\bigskip

\example{ex: crystal}
Assume that $n=4$ and that $\theta (a)^2 = \theta (b)^2 = -1$. Then $\theta (a)$ is a primitive $2$-root of unity and we have $\theta (b)^2= - \theta (a)^0$. 
If $0<\xi_1 < 2 < \xi_2 < 4$ (i.e. $r_1 \in \{0,1\}$ and $r_2\in \{2,3\}$), 
then we have
$$\operatorname{Irr}(\mathbb{C} \mathcal{H}_4 )=\{D_{\mu}^{\xi_1}\ |\ \mu\in \Bip_2^{(0,0)}(4)\}.$$
$$\operatorname{Irr}(\mathbb{C} \mathcal{H}_4 )=\{D_{\mu}^{\xi_2}\ |\ \mu\in \Bip_2^{(2,0)}(4)\}.$$

To compute $\Bip_2^{(0,0)}(4)$ (respectively $\Bip_2^{(2,0)}(4)$), we compute 
the crystal graph associated to the action of 
the quantum group $\mathcal{U}_v (\widehat{\mathfrak{sl}}_2)$ on the 
Fock space $\mathcal{F}^{(0,0)}$ (respectively $\FC^{(2,0)}$). 
The submodule generated by the empty bipartition gives an irreducible 
highest weight module $\mathcal{M}[0,0]$ (respectively $\MC[2,0]$) 
with highest weight $2\Lambda_0$. 
We only give the part of both crystals containing the bipartitions 
up to rank $4$. They are computed as explained in 
\cite[2.1, 2.2]{Jbijection}:

\begin{center} 
\begin{picture}(500,170)
\put( 50, 140){$(\emptyset,\emptyset)$} 
\put( 50, 100){$(1,\emptyset)$} 
\put(100,  60){$(2,\emptyset)$}
\put(  0,  60){$(1,1)$} 
\put(150,  20){$(3,\emptyset)$}
\put( 50,  20){$(2.1,\emptyset)$}
\put(  0,  20){$(2,1)$}
\put(  0, -15){$(2,2)$}
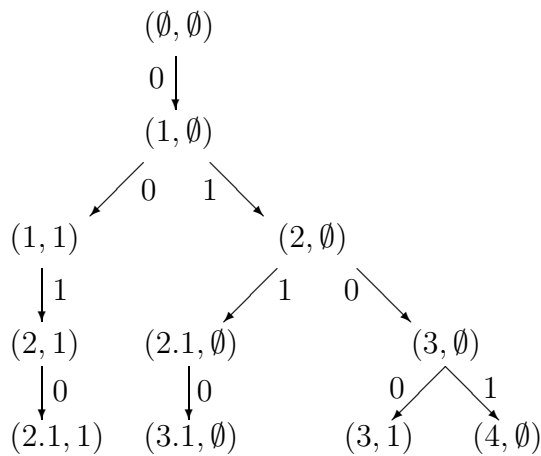
\put( 50, -15){$(3.1,\emptyset)$}
\put(125,-15){$(3,1)$} 
\put(173,-15){$(4,\emptyset)$}

\put( 12,  15){\vector(0,-1){20}} 
\put( 67,  15){\vector(0,-1){20}}
\put(163,  15){\vector(-1,-1){20}} 
\put(163,  15){\vector(1,-1){20}}
\put( 62, 132){\vector(0,-1){20}} 
\put( 50, 92){\vector(-1,-1){20}} 
\put( 75, 92){\vector(1,-1){20}} 
\put( 12, 52){\vector(0,-1){20}}
\put(130, 52){\vector(1,-1){20}} 
\put(100, 52){\vector(-1,-1){20}} 
 
\put( 52, 120){$0$} 
\put( 72, 78){$1$} 
\put( 49, 78){$0$}  
\put(125, 40){$0$} 
\put( 16, 40){$1$} 
\put(100, 40){$1$} 
\put(177, 3){$1$} 
\put( 16, 2){$0$} 
\put( 70, 2){$0$} 
\put(142, 3){$0$} 

\put(28,165){{\bf Crystal graph of ${\boldsymbol{\MC[0,0]}}$}}


\put(  275, 140){$(\emptyset,\emptyset)$} 
\put( 275, 100){$(1,\emptyset)$} 
\put( 325, 60){$(2,\emptyset)$}
\put( 225, 60){$(1,1)$} 
\put(375,20){$(3,\emptyset)$}
\put( 275, 20){$(2.1,\emptyset)$}
\put( 225, 20){$(2,1)$} 
\put( 225,-15){$(2.1,1)$} 
\put( 275,-15){$(3.1,\emptyset)$} 
\put( 350,-15){$(3,1)$}
\put( 398,-15){$(4,\emptyset)$} 

\put( 287, 132){\vector(0,-1){20}} 
\put( 275, 92){\vector(-1,-1){20}} 
\put( 300, 92){\vector(1,-1){20}} 
\put( 237, 52){\vector(0,-1){20}}
\put( 355, 52){\vector(1,-1){20}} 
\put( 325, 52){\vector(-1,-1){20}} 
\put( 237, 15){\vector(0,-1){20}} 
\put( 292, 15){\vector(0,-1){20}}
\put( 388, 15){\vector(-1,-1){20}} 
\put( 388, 15){\vector(1,-1){20}}

\put(  277, 120){$0$} 
\put(  297, 78){$1$} 
\put(  274, 78){$0$} 
\put(  350, 40){$0$} 
\put(  241, 40){$1$} 
\put(  325, 40){$1$} 
\put(  402, 3){$1$} 
\put(  241, 2){$0$} 
\put(  295, 2){$0$} 
\put(  367, 3){$0$} 

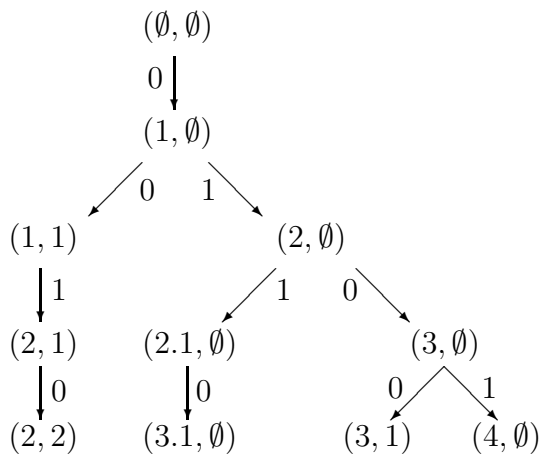
\put(250,165){{\bf Crystal graph of ${\boldsymbol{\MC[2,0]}}$}}

\end{picture} 
\end{center}

\vskip1cm

By Theorem \ref{main theorem}, we then have
\eqna
\operatorname{Irr}(  \mathbb{C} \mathcal{H}_4      )&=&
\{D^{\xi_1}_{(4,\emptyset)}, 
D^{\xi_1}_{(31,\emptyset) }, D^{\xi_1}_{(3,1)}, D^{\xi_1}_{(2,2)}   \}\\
&=&\{D^{\xi_2}_{(4,\emptyset)}, 
D^{\xi_2}_{(31,\emptyset) }, D^{\xi_2}_{(3,1)}, D^{\xi_2}_{(2,2)}   \}
\endeqna
and, by Theorem \ref{bijection}, we have 
\eqna
D^{\xi_1}_{(4,\emptyset)} &\simeq& D^{\xi_2}_{(4,\emptyset)}, \\ 
D^{\xi_1}_{(31,\emptyset) } &\simeq&  D^{\xi_2}_{(31,\emptyset) },\\ 
 D^{\xi_1}_{(3,1)}&\simeq&  D^{\xi_2}_{(3,1) },\\
 D^{\xi_1}_{(2,2)}&\simeq&  D^{\xi_2}_{(2.1,1) }.
\endeqna
\bigskip

\noindent {\bf Acknowledgements.} We would like to thank Bernard Leclerc and Meinolf Geck for many useful discussions on the topic of this paper.  Both authors are supported by the ``Agence national de la recherche" No. JCO7 192339.

\bigskip



\begin{thebibliography}{131}

\bibitem{arikillt}  
\textsc{S.\ Ariki,}
On the decomposition numbers of the Hecke algebra of $G(m, 1, n)$, 
{\it J. Math. Kyoto Univ.} {\bf 36} (1996), 789--808.


\bibitem{arikilivre}  
\textsc{S.\ Ariki,} 
{\it Representations of Quantum algebras and
combinatorics of Young tableaux}, 
University Lecture Series \textbf{26}, Amer.\ Math.\ Soc., Providence, RI, 2002.

\bibitem{aj}
{\sc S.~Ariki and N.~Jacon}, Dipper James Murphy's conjecture for Hecke algebras of type $B_n$, 
to appear in Progress in Mathematics (Birkh\"auser).

\bibitem{akt}
{\sc S.~Ariki, V.~Kreiman and S.Tsuchioka}, 
On the tensor product of two basic representations of $U_v(\widehat{\sG\lG}_e)$, 
 to appear in Adv. Math.

\bibitem{bilatere}
{\sc C. Bonnaf\'e}, 
Two-sided cells in type B (asymptotic case), 
{\it J. Algebra} {\bf 304} (2006), 216--236.

\bibitem{BGIL} 
{\sc C. Bonnaf\'e, M. Geck, L. Iancu and T. Lam}, 
{\it On domino insertion and Kazhdan-Lusztig cells in type $B_n$}, 
to appear in Progress in Mathematics (Birkh\"auser).

\bibitem{iancu} {\sc C. Bonnaf\'e and L. Iancu}, 
Left cells in type $B_n$ with unequal parameters, 
{\it Representation Theory} {\bf 7} (2003), 587--609.

\bibitem{CL}
{\sc C. Carr\'e and B. Leclerc}, 
Splitting the square of a Schur function into its symmetric and 
antisymmetric parts, 
{\it J. Algebraic Combin.} {\bf 4} (1995), 201--231.

\bibitem{DJMa}
{\sc R.~Dipper, G. D.~James and A.~Mathas}, 
Cyclotomic $q$-Schur algebras,
{\it Math. Z.} {\bf 229} (1998), 385--416.

\bibitem{DJMU}
{\sc R. Dipper, G. D. James and G. E. Murphy}, 
Hecke algebras of type $B_n$ at roots of unity, 
{\it Proc. London Math. Soc.} {\bf 70} (1995), 505--528.


\bibitem{duscott}
{\sc J. Du and L. Scott}, 
The $q$-${\rm Schur}\sp 2$ algebra, 
{\it  Trans. Amer. Math. Soc.} {\bf   352}  (2000),  no. 9, 4325--4353.



\bibitem{FLOTW}
\textsc{O. Foda, B. Leclerc, M. Okado, J-Y Thibon and T. Welsh,}
Branching functions of $A\sp {(1)}\sb {n-1}$ 
and Jantzen-Seitz problem for Ariki-Koike   algebras,
{\it Adv. Math.} \textbf{141} no. 2 (1999), 322--365.

\bibitem{ggor}
\textsc{V. Ginzburg, N. Guay, E. Opdam and R. Rouquier},
 On the category $\mathcal{O}$ for rational Cherednik algebras,  
{\it  Invent. Math.}  154  (2003),  no. 3, 617--651.



\bibitem{geckLausanne}  
\textsc{M.\ Geck,} 
Modular representations of Hecke algebras, 
In: {\it Group representation theory} (EPFL, 2005; eds. M. Geck, D.
Testerman and J. Th\'{e}venaz), EPFL Press (2007), p. 301-353.



\bibitem{geckcell}  
\textsc{M.\ Geck,} 
Hecke algebras of finite type are cellular,
{\it Invent Math.} {\bf 169} (2007), 501--517.



\bibitem{geck-iancu}
{\sc M. Geck and L. Iancu}, 
Lusztig's $\ab$-function in type $B_n$ in the asymptotic case, 
{\it Nagoya Math. J.} {\bf 182} (2006), 199--240.

\bibitem{gj}
{\sc M.~Geck and N.~Jacon}, 
Canonical basic sets in type $B_n$,  
{\it J. Algebra} {\bf 306}  (2006), 104--127.


\bibitem{GIP} 
{\sc M. Geck, L. Iancu and C. Pallikaros}, 
Specht modules and Kazhdan--Lusztig cells in type $B_n$, 
to appear in {\it J. Pure and Applied Algebra}, Volume 212, Issue 6, (2008), 1310--1320.

\bibitem{gecklivre}
{\sc M.~Geck and G.~Pfeiffer}, 
{\it Characters of finite Coxeter groups and
Iwahori--Hecke algebras}, London Math. Soc. Monographs, New Series {\bf 21},
Oxford University Press, New York 2000. xvi+446 pp.

\bibitem{GL}
{\sc J.J.~Graham and G.I.~Lehrer},
 Cellular algebras, 
{\it Invent. Math.} {\bf 123} (1996), 1--34.


\bibitem{gordon}
{\sc I. Gordon}, 
Quiver varieties, category $\mathcal{O}$ for rational Cherednik algebras, and Hecke algebras, 
preprint, available at  \url{http://arxiv.org/abs/math/0703150}.


\bibitem{gordon martino}
{\sc I. Gordon and M. Martino}
Calogero-Moser space, reduced rational Cherednik algebras, and two-sided cells,
preprint, available at  \url{http://arxiv.org/abs/math/0703153}



\bibitem{algo}  \textsc{N.\ Jacon,} 
An algorithm for the computation of the 
decomposition matrices for Ariki-Koike algebras,
{\it J. Algebra} {\bf 292}  (2005), 100--109.


\bibitem{Jbijection}  \textsc{N.\ Jacon,} Crystal graphs of irreducible highest weight $\Ue$-modules of level two and Uglov bipartitions, 
 {\it J. Algebr Comb.} {\bf 27}  (2008), 143--162.


\bibitem{JLbijection}  
\textsc{N.\ Jacon and C.\ Lecouvey}  
Crystal isomorphisms for irreducible highest weight $\Ue$-modules of higher level, 
preprint (2007) available at \url{http://arxiv.org/abs/0706.0680}.




\bibitem{JMMO}
{\sc M.~Jimbo, K.~C.~Misra, T.~Miwa and M.~Okado}, 
Combinatorics of representations of $U_q(\hat{\mathfrak{sl}}(n))$ at $q=0$, 
{\it Comm. Math. Phys.} {\bf 136} (1991), 543--566.





\bibitem{lam} {\sc T. Lam}, 
Growth diagrams, domino insertion, and sign-imbalance, 
{\it Journal of Combinatorial Theory} {\bf 107} (2004), 87--115.

\bibitem{van} {\sc M. vanLeeuwen}, 
The Robinson-Schensted and Schutzenberger algorithms, an elementary approach, 
The Foata Festschrift, {\it Electron. J. Combin.} {\bf 3} (1996), 
Research Paper 15.

\bibitem{llt} 
{\sc A. Lascoux, B. Leclerc and J-Y Thibon},
Hecke algebras at roots of unity and crystal bases of quantum affine algebras.
{\it Comm. Math. Phys.} {\bf 181} (1996), 205-263.

\bibitem{lusztig}
{\sc G. Lusztig}, 
{\it Hecke algebras with unequal parameters}, 
CRM Monographs Ser. {\bf 18}, Amer. Math. Soc., Providence, RI, 2003.


\bibitem{mathaslivre}  
\textsc{A.\ Mathas}, 
{\it Iwahori-Hecke algebras and Schur
algebras of the symmetric group}, 
University Lectures Series, AMS, Providence, \textbf{15}, 1999.

\bibitem{white}
{\sc M. Shimozono and D. E. White}, 
Color-to-spin ribbon Schensted algorithms,
{\it Discrete Math.} {\bf 246} (2002), 295--316.

\bibitem{uglov} 
\textsc{D.\ Uglov},
Canonical bases of higher-level $q$-deformed Fock spaces and 
Kazhdan-Lusztig polynomials,
Kashiwara, Masaki (ed.) et al., Boston: Birkh\"auser. 
{\it Prog. Math.} \textbf{191}  (2000): 249--299.




\bibitem{yvonne} \textsc{X.\ Yvonne},
 A conjecture for $q$-decomposition matrices of cyclotomic $v$-Schur algebras. 
{\it  J. Algebra}  304  (2006),  no. 1, 419--456.

\bibitem{xavierthese} \textsc{X.\ Yvonne},
{\it Base canonique d'espaces de Fock de niveau sup\'erieur}, 
Ph.D. Thesis, Universit\'e de Caen, 
available at \url{http://tel.archives-ouvertes.fr/tel-00137705}.

\bibitem{xavieralgo} \textsc{X.\ Yvonne},
An algorithm for computing the canonical bases of 
higher-level $q$-deformed Fock spaces,
{\it J. Algebra} \textbf{309} (2007), 760--785.


\end{thebibliography}
\end{document}